\documentclass[12pt]{amsart}

\usepackage{amsmath}
\usepackage{amssymb}
\usepackage{amscd}

\usepackage{enumitem}
\usepackage{MnSymbol}
\usepackage{bbold}
\usepackage{tikz}
\usepackage{verbatim}
\usepackage{eucal}
\usepackage{nicefrac}

\newcommand{\mc}{\mathcal}
\newcommand{\mb}{\mathbb}

\DeclareMathOperator*{\argmin}{argmin}

\begin{document}

\newtheorem{tem}{Theorem}
\newtheorem{lem}[tem]{Lemma}
\newtheorem{cor}[tem]{Corollary}

\title[Realizability as Matrix Completion]{Realizability of Polytopes as a Low Rank Matrix Completion Problem}
\author{Michael Gene Dobbins}
\address{ GAIA \\ POSTECH \\ Daejeon and Pohang, South Korea }
\email{michaelgenedobbins@gmail.com}

\begin{abstract}
This article gives necessary and sufficient conditions for a relation to be the containment relation between the facets and vertices of a polytope.  Also given here, are a set of matrices parameterizing the linear moduli space and another set parameterizing the projective moduli space of a combinatorial polytope. 
\end{abstract}
\maketitle

\section{Introduction}

Given a polytope, there are various kinds of data we could use to describe it.  The purpose of this article is to go in the other direction.  Given some data, determine if there exists a polytope realizing this data.  
An early realizability result for polytopes was given by Steinitz, who showed that a graph gives the edges and vertices of some 3 dimensional polytope if and only if the graph is planer and 3-connected \cite{steinitz}. 
In higher dimensions, 
there are known impediments to generalizing this to purely combinatorial conditions for when a poset can be realized as the face lattice of a polytope \cite{RG}. 
There are, however, known algebraic conditions for determining when a given poset is realizable, which come directly from the definition of a face \cite{grun}. 
When such a realization exists, we say the poset is a combinatorial polytope, and say 
the class of such polytopes have the same combinatorial type.

The main result of this article, 
Theorem \ref{tem:rmcp} in Section \ref{sec:fvm}, gives necessary and sufficient conditions for any given relation to be the containment relation between the facets and vertices of a polytope. 
These conditions come in two parts, a combinatorial part and an algebraic part.  
For the combinatorial part, we present a combinatorial construction to get the face lattice back from the facet-vertex incidence 
and give conditions in terms of this lattice. 
For the algebraic part, 
we associate a matrix to any centered polytope, which we call the facet-vertex matrix.  The entries of this matrix are the inner product of the vertices of the polytope with the vertices of its polar polytope.  
Theorem \ref{tem:rmcp} gives conditions for a matrix to be the facet-vertex matrix of the desired polytope, and shows that facet-vertex matrices parameterize the space of centered realizations of a combinatorial type of polytope modulo linear transformations, which we call the linear moduli space. 
In contrast to the realizability conditions we get from the definition of a face, Theorem \ref{tem:rmcp} gives simpler algebraic conditions at the expense of adding some combinatorial conditions.

For any polytope, we have the cone of homogeneous coordinates over the polytope in a space one dimension higher, and we can get a projectively equivalent polytope back from that cone by intersecting it with an appropriately chosen affine hyperplane.  This makes the problem of realizing a combinatorial type of polytope equivalent to the problem of realizing the corresponding combinatorial type of cone.  
We associate a similarly defined facet-ray matrix to any polytopal pointed cone, and in this way, Theorem \ref{tem:rmcc} in Section \ref{sec:fvm}, will give conditions for a relation to be the facet-vertex incidence of a polytope in terms of a facet-ray matrix of a cone over that polytope, and shows that such matrices parameterize the projective moduli space of a combinatorial type of polytope, its realization space modulo projective transformations.

The algebraic conditions of Theorem \ref{tem:rmcp} 
are the same as the conditions Robertson gave 
for when a small perturbation of a polytope's facet-vertex matrix is the facet-vertex matrix of a perturbed polytope of the same combinatorial type, 
which he used to count the dimension of a polytope's realization space \cite{robertson}.  
The algebraic conditions of Theorem \ref{tem:rmcc} are the same as what D\'iaz gave to characterize the facet-ray matrices of cones of a given a combinatorial type. 
The results of Robertson and D\'iaz, however, both assume we start with a combinatorial polytope,   
so we cannot use these to decide if some given combinatorial data can be realized.

The combinatorial conditions appearing in Theorem \ref{tem:rmcp} are given in terms of the Dedekind-MacNeille completion of the facets and vertices \cite{introlat}, which we call the maxbiclique lattice.  These conditions are related to those of abstract polytopes,  
which are posets that satisfy certain properties that hold for all combinatorial polytopes.  
Abstract polytopes are presented in \cite{arp}, and are generally of interest for the symmetry groups of regular abstract polytopes. 
If one wants to consider purely combinatorial objects resembling a combinatorial polytope, there are more faithful objects one may consider, such as shellable lattices \cite{zieg} or Eulerian lattices \cite{enumcomb}.  
Many other related constructions can be found 
in \cite{paffen}, which presents combinatorial operations on a hierarchy of such objects, 
but here we are only interested in the minimal properties we need for Theorem $\ref{tem:rmcp}$.  These are really properties of a poset's flag graph, which have been characterized by Peterin \cite{peterin}.

The motivation D\'iaz gave for characterizing the facet-ray matrix of a combinatorial type of cone was for use as a lemma to characterize the Gramian matrix of combinatorial polytopes in other geometric spaces, such as spherical or hyperbolic space.  In Section \ref{sec:gram}, Theorem \ref{tem:rmcc} is used instead to give realizability conditions in terms of a polytope's Gramian. 
The Gramian has the geometric appeal that its entries 
are the cosine of the dihedral angles of the polytope.

Section \ref{sec:def1} gives basic definitions needed for Section \ref{sec:fvm} with new, less common, or otherwise ambiguous terms in bold.  
Section \ref{sec:fvm} gives necessary and sufficient conditions for a relation to be realizable and shows that the facet-vertex(ray) matrices parameterize the moduli spaces of polytopes. Section \ref{sec:def2} gives more definitions needed for Section \ref{sec:gram}, which gives other realizability conditions in terms of the Gramian of a polytope in a geometric space.  Finally, Section \ref{sec:rem} ends with some remarks about complexity and the Gale dual.

\section{Definitions}\label{sec:def1}

A poset is called a lattice when every pair of elements $i,j$ has a meet $i \wedge j$ and a join $i\vee j$,  which are respectively the greatest lower bound and least upper bound of the pair. 
A meet irreducible is an element of a lattice that cannot be expressed as the meet of other elements, and a join irreducible is the order dual of that.  
The greatest and least elements are denoted $\top$ and $\bot$ respectively. 
The faces $f$ of a polytope $P$ are subsets of $P$ where some affine inequality $\left<a,x\right> \leq b$ that is satisfied for every $x\in P$, holds as an equality $f = \{x \in P: \left<a,x\right> = b\}$, and the face lattice of $P$ is the poset of its faces ordered by containment.  
The meet and join irreducibles of the face lattice are the facets and vertices respectively.

Here we consider a relation $\mc{R} \subset I\times J$ to be a subset of the product of two sets and, where convenient, use terminology for bipartite graphs. 
A {\bf filled $x$-incidence matrix} $M$ of a relation ${\mc{R} \subset I\times J}$ between row indices $I$ and column indices $J$, has $[M]_{i,j} = x$ for $(i,j) \in \mc{R}$ and $[M]_{i,j} < x$ for $(i,j) \not\in \mc{R}$. 
An {\bf induced biclique} $(I_0,J_0)$ of a relation is a pair of sets $I_0 \subset I$, $J_0 \subset J$ such that every element of one part is incident to every element of the other part ${I_0 \times J_0 \subset \mc{R}}$.  A {\bf maxbiclique} of a relation is a maximal induced biclique.  The {\bf maxbiclique lattice} of a relation with one part $J$ specified as the lower is the poset consisting of maxbicliques ordered by containment of elements in the lower part.
$(I_0,J_0) \leq (I_1,J_1) := J_0 \subset J_1$. 
Lemma \ref{lem:mbc} in the next section says that finding the irreducibles of a finite lattice and finding the maxbiclique lattice of the incidence relation between irreducibles are inverse operations.

For now, a polytope is the convex hull of finitely many points in $\mb{R}^d$.  In Section \ref{sec:def2}, we will refer to more general objects as polytopes.  A polytope is centered when it contains the origin, and 
the {\bf covertices} of a full-dimensional centered polytope $P$ are the vertices of its polar polytope 
\[{P^* = \{ y:\: \forall x\in P, \left<y,x\right> \leq 1 \}}.\] 
We may think of the polar as consisting of vectors in dual space representing half-spaces that contain $P$, 
and the covertices as representing the supporting hyperplanes of the facets.  
A full-dimensional centered polytope $P$ with vertices $w_j$ and covertices $h_i$, is given by 
\[P =\left\{ \sum_{j\in J} t_j w_j :\: t_j \geq 0, \sum_{j\in J} t_j = 1 \right\} = \{ x :\: \forall i \in I, \left<h_i,x\right> \leq 1 \},  \] 
and the {\bf facet-vertex matrix} $M$ of $P$ 
has entries $[M]_{i,j}= \left<h_i,w_j \right>$.  Figure \ref{fig:fvm} gives an example of this matrix.  
Note that the transpose of the facet-vertex matrix of a centered polytope is the facet-vertex matrix of its polar polytope, and that we use the convention that rows correspond to covertices and columns to vertices, since these are respectively covariant and counter-variant under change of coordinates.  
A type $\mc{R}$ polytope has facets and vertices indexed by $I$ and $J$ and facet-vertex incidence $\mc{R}$, and its realization space consists of $x\in\mb{R}^{|J|\times d}$ such that $([x]_{j,1},\dots,[x]_{j,d})$ is the $j^{\rm th}$ vertex of a type $\mc{R}$ polytope. 
Clearly the facet-vertex matrix of a type $\mc{R}$ polytope is a rank $d$ filled 1-incidence matrix of $\mc{R}$.  Theorem \ref{tem:rmcp} will give conditions on $\mc{R}$ for a filled 1-incidence matrix to be a facet-vertex matrix. 

\begin{figure}[h]
\begin{center}
\begin{tikzpicture}

\begin{scope}[scale=1.5]

\path
 (0,0.5) coordinate (o)
 (-1.3,-0.9) coordinate (a) 
 (-0.7,-0.1) coordinate (b)
 (1.3,-0.1) coordinate (c)
 (0.7,-0.9) coordinate (d)
 (0,1.5) coordinate (e)
 (-1,1.5) coordinate (l); 

\path
 (2,1.5) coordinate (f)
 (0.6,2.3) coordinate (g)
 (-2,1.5) coordinate (h)
 (-0.6,0.7) coordinate (i)
 (0,-0.5) coordinate (j);

\end{scope}

\matrix[ampersand replacement=\&, row sep=10pt, column sep=-40pt]
{

\draw[very thick]
 (a) -- (b) -- (c) -- (d) -- cycle
 (a) -- (e)
 (b) -- (e)
 (c) -- (e)
 (d) -- (e); 

\node [below, xshift=16pt] at (a) {\small (-1,-1,-1)};
\node [below, xshift=16pt] at (b) {\small (-1,1,-1)};
\node [below, xshift=16pt] at (c) {\small (1,1,-1)};
\node [below, xshift=16pt] at (d) {\small (1,-1,-1)};
\node [above] at (e) {\small (0,0,1)}; \&

\draw[very thick]
 (f) -- (g) -- (h) -- (i) -- cycle
 (f) -- (j)
 (g) -- (j)
 (h) -- (j)
 (i) -- (j);

\draw
 (a) -- (b) -- (c) -- (d) -- cycle
 (a) -- (e)
 (b) -- (e)
 (c) -- (e)
 (d) -- (e); 

\node [above, xshift=-32pt, yshift=-6pt] at (f) {\small (2,0,1)};
\node [above] at (g) {\small (0,2,1)};
\node [above, xshift=-6pt, yshift=2pt] at (h) {\small (-2,0,1)};
\node [above, xshift=-4pt, yshift=8pt] at (i) {\small (0,-2,1)};
\node [below] at (j) {\small (0,0,-1)}; \\

\draw[very thick, fill=gray!50]
 (c) -- (d) -- (e) -- cycle;

\draw 
 (f) -- (g) -- (h) -- (i) -- cycle
 (f) -- (j)
 (g) -- (j)
 (h) -- (j)
 (i) -- (j);

\draw
 (a) -- (b) -- (c) -- (d) -- cycle
 (a) -- (e)
 (b) -- (e)
 (c) -- (e)
 (d) -- (e);

\draw[very thick, ->]
 (o) -- (f); \&

\node at (o) 
{$\left[\begin{array}{ccccc}
\bf 1 & \bf 1 & \bf -3 & \bf -3 & \bf 1 \\
-3 & 1 & 1 & -3 & 1 \\
-3 & -3 & 1 & 1 & 1 \\
1 & -3 & -3 & 1 & 1 \\
1 & 1 & 1 & 1 & -1 \\
\end{array}\right]$}; \\
};

\end{tikzpicture}

\caption{
Top from left, a 3-polytope and its polar with vertex coordinates; bottom, a covertex with the corresponding facet shaded and the facet-vertex matrix with the corresponding row in bold.}\label{fig:fvm}
\end{center}
\end{figure}
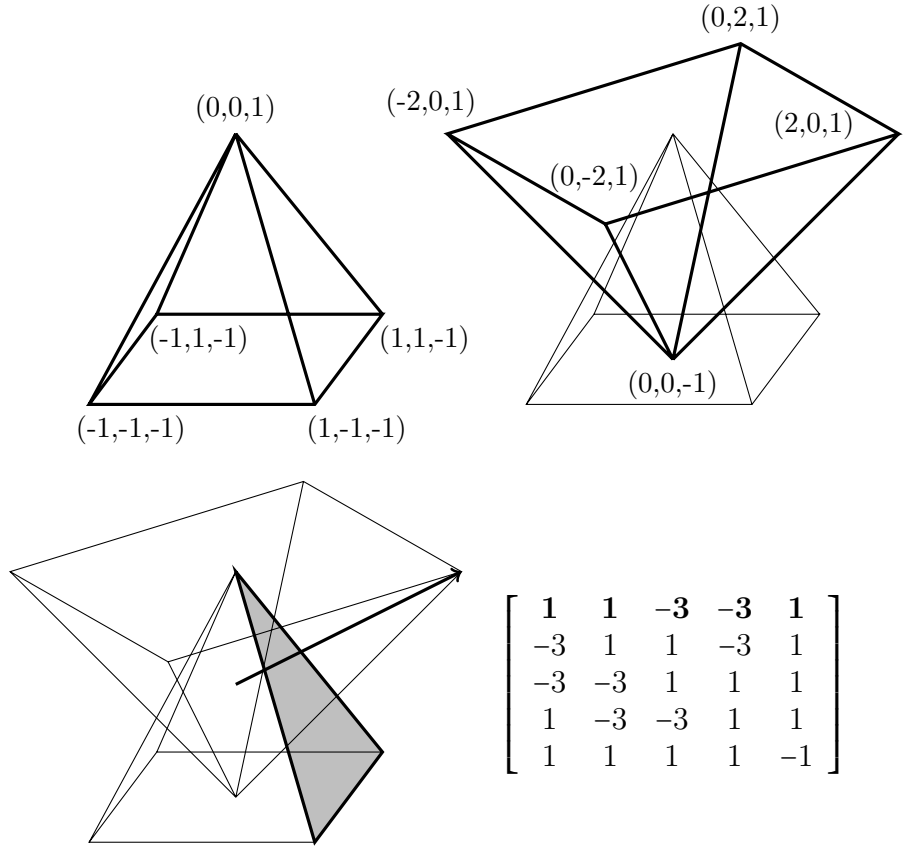

A {\bf cone} $C$ is the set of positive linear combinations of vectors in $\mb{R}^{d+1}$ that does not contain a line.  That is, by `cone' we mean what may be called a pointed polytopal cone with vertex at the origin.  A set of generators of $C$ is a minimal set of vectors such that $C$ consists of positive linear combinations of these vectors, and a set of {\bf cogenerators} is a set of generators of its polar cone 
\[C^* = \{ y:\: \forall x\in C, \left<y,x\right> \leq 0 \}. \] 
A full-dimensional cone $C$ with generators $w_j$ and cogenerators $h_i$, is given by 
\[ C= \left\{ \sum_{j\in J} t_j w_j :\: t_j \geq 0 \right\} = \{ x :\: \forall i\in I, \left<h_i,x\right> \leq 0 \},  \] 
and a {\bf facet-ray matrix} $N$ of $C$ 
has entries $[N]_{i,j} = \left<h_i,w_j\right>$.  
Note that the generators and cogenerators of a cone are defined up to positive scaling of each vector.  As such, a cone does not have a unique facet-ray matrix. 

\begin{figure}[h]
\begin{center}
\begin{tikzpicture}

\matrix[row sep=10pt, column sep=10pt]
{

\node[below] at (l) {a)};

\fill[gray!50]
 (c) -- (d) -- (e) -- cycle;

\draw
 (a) -- (b) -- (c) -- (d) -- cycle
 (a) -- (e)
 (b) -- (e)
 (c) -- (e)
 (d) -- (e); 

\draw[ultra thick]
 (d) -- (e);

\fill
 (d) circle (0.1); &

\node[below] at (l) {b)};

\fill[gray!50]
 (c) -- (d) -- (e) -- cycle;

\draw
 (a) -- (b) -- (c) -- (d) -- cycle
 (a) -- (e)
 (b) -- (e)
 (c) -- (e)
 (d) -- (e); 

\draw[ultra thick]
 (d) -- (e);

\fill
 (e) circle (0.1); \\

\node[below] at (l) {c)};

\fill[gray!50]
 (c) -- (d) -- (e) -- cycle;

\draw
 (a) -- (b) -- (c) -- (d) -- cycle
 (a) -- (e)
 (b) -- (e)
 (c) -- (e)
 (d) -- (e); 

\draw[ultra thick]
 (d) -- (c);

\fill
 (d) circle (0.1); &

\node[below] at (l) {d)};

\fill[gray!50]
 (a) -- (d) -- (e) -- cycle;

\draw
 (a) -- (b) -- (c) -- (d) -- cycle
 (a) -- (e)
 (b) -- (e)
 (c) -- (e)
 (d) -- (e); 

\draw[ultra thick]
 (d) -- (e);

\fill
 (d) circle (0.1); \\
};

\end{tikzpicture}

\begin{tikzpicture}

\begin{scope}[scale=2/3]

\path
 (-2,-1) coordinate (eab)
 (-2,1)  coordinate (eba)
 (-1,2)  coordinate (ebc)
 (1,2)   coordinate (ecb)
 (2,1)   coordinate (ecd)
 (2,-1)  coordinate (edc)
 (1,-2)  coordinate (eda)
 (-1,-2) coordinate (ead)

 (4,6) coordinate (cbj)
 (3,5) coordinate (cbe)
 (3,4) coordinate (ceb)
 (4,3) coordinate (ced)
 (5,3) coordinate (cde)
 (6,4) coordinate (cdj)

 (4,-6) coordinate (daj)
 (3,-5) coordinate (dae)
 (3,-4) coordinate (dea)
 (4,-3) coordinate (dec)
 (5,-3) coordinate (dce)
 (6,-4) coordinate (dcj)

 (-4,-6) coordinate (adj)
 (-3,-5) coordinate (ade)
 (-3,-4) coordinate (aed)
 (-4,-3) coordinate (aeb)
 (-5,-3) coordinate (abe)
 (-6,-4) coordinate (abj)

 (-4,6) coordinate (bcj)
 (-3,5) coordinate (bce)
 (-3,4) coordinate (bec)
 (-4,3) coordinate (bea)
 (-5,3) coordinate (bae)
 (-6,4) coordinate (baj);

\end{scope}

\draw[very thick, loosely dotted]
 (abe) -- (bae)
 (bce) -- (cbe)
 (cde) -- (dce)
 (dae) -- (ade)

 (abj) -- (baj)
 (bcj) -- (cbj)
 (cdj) -- (dcj)
 (daj) -- (adj)

 (eab) -- (aeb)
 (eba) -- (bea)
 (ebc) -- (bec)
 (ecb) -- (ceb)
 (ecd) -- (ced)
 (edc) -- (dec)
 (eda) -- (dea)
 (ead) -- (aed);

\draw[thick]
 (abe) -- (aeb)
 (bae) -- (bea)
 (bce) -- (bec)
 (cbe) -- (ceb)
 (cde) -- (ced)
 (dce) -- (dec)
 (dae) -- (dea)
 (ade) -- (aed)

 (abj) -- (adj)
 (bcj) -- (baj)
 (cdj) -- (cbj)
 (daj) -- (dcj)

 (eab) -- (eba)
 (ebc) -- (ecb)
 (ecd) -- (edc)
 (eda) -- (ead);

\draw[line width=0.2cm, gray!50, line cap=round]
 (abe) -- (abj)
 (bae) -- (baj)
 (bce) -- (bcj)
 (cbe) -- (cbj)
 (cde) -- (cdj)
 (dce) -- (dcj)
 (dae) -- (daj)
 (ade) -- (adj)

 (aeb) -- (aed)
 (bec) -- (bea)
 (ced) -- (ceb)
 (dea) -- (dec)

 (eab) -- (ead)
 (ebc) -- (eba)
 (ecd) -- (ecb)
 (eda) -- (edc);

\fill
 (dec) circle (0.1)
 (edc) circle (0.1)
 (dce) circle (0.1)
 (dea) circle (0.1);

\node[above, xshift= 1pt] at (dec) {a};
\node[left, yshift= 4pt] at (edc) {b};
\node[right, yshift= 2pt] at (dce) {c};
\node[left, yshift= -2pt] at (dea) {d};

\end{tikzpicture}

\caption{
Top, 4 flags of a polytope; bottom, the polytope's flag graph with edge between flags that differ by a vertex, edge, and facet shown dotted, thin, and thick gray respectively with vertices corresponding to the 4 flags indicated.}\label{fig:flags}
\end{center}
\end{figure}
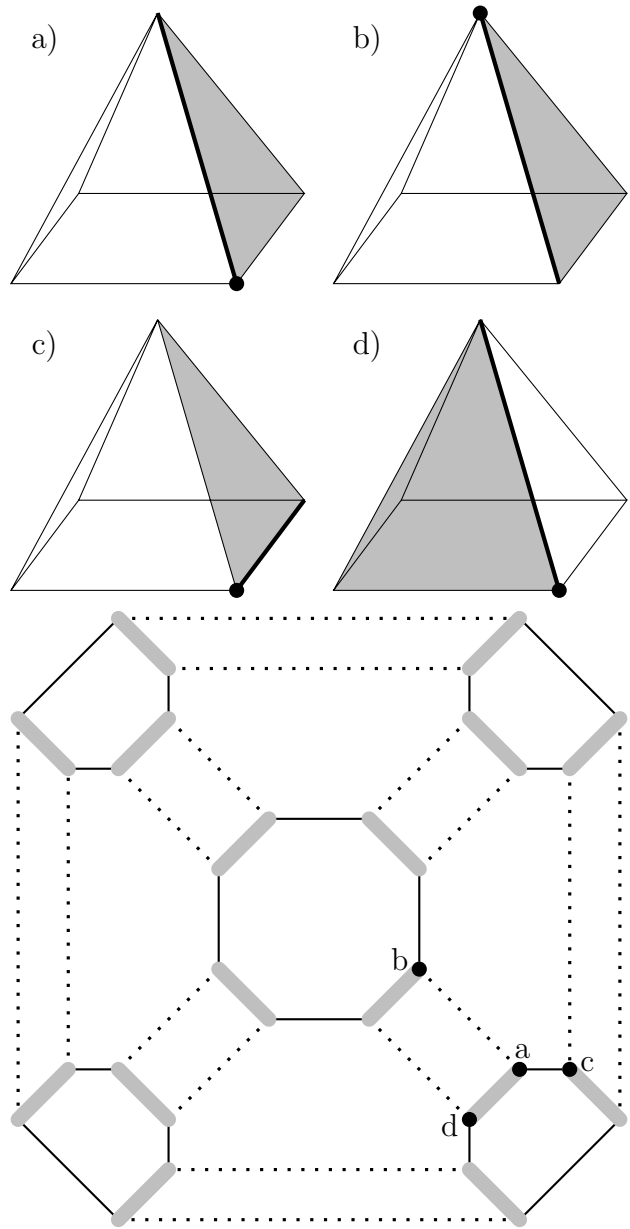

Now we will define the properties that will be used in Theorem \ref{tem:rmcp} for the maxbiclique lattice of a relation. 
A {\bf flag} of a lattice, or more generally of a poset, is a maximal totally ordered subset, and implicit in each of these properties is that the poset be graded, which means 
all flags have the same length.  The {\bf rank} of a graded poset is the length of its flags.  Note that we begin counting at 0, so the rank of the singleton poset is 0.  The rank of an element of a poset is the number of elements below it in a flag.  The {\bf flag graph} of a poset is a graph consisting of a node for each flag $\mathcal{F}$, and edges between pairs of flags that differ in only one element.  That is, $\mc{F}$ and $\mc{F}'$ are neighbors when $|\mc{F}\setminus \mc{F}'| = 1$. 
Figure \ref{fig:flags} shows the flag graph of a polytope, which is the flag graph of its face lattice. 
When the flag graph of a graded poset is connected we say the poset is {\bf flag connected}.  
An interval $[a,b]$ of a poset is the set of all elements between two particular elements $a,b$ of the poset, $[a,b] = \{c:\: a \leq c \leq b\}$.  
Finally, a {\bf diamond} poset is a graded poset where every rank 2 interval has 4 elements

\section{Realizability and the Facet-Vertex(Ray) Matrix}\label{sec:fvm}

Here is the main theorem for polytopes.

\begin{tem}\label{tem:rmcp}
Given a relation $\mc{R}$, there is a polytope with facet-vertex incidence $\mc{R}$ if and only if both the maxbiclique lattice of $\mc{R}$ is 
flag connected, rank $d+1$, and diamond, and $\mc{R}$ has 
a rank $d$ filled 1-incidence matrix. 

Moreover, 
for such $\mc{R}$, $M$ is the facet-vertex matrix of a centered type $\mc{R}$ polytope if and only if $M$ is a rank $d$ filled 1-incidence matrix of $\mc{R}$, and these matrices parameterize the linear moduli space of such polytopes. 
\end{tem}

From this we can see that the dimension of the realization space of a combinatorial polytope is $d(n+m) - |\mc{R}|$, where $n$ and $m$ are the number of facets and vertices respectively, and $|\mc{R}|$ is the number pairs of these that are incident \cite{robertson}, which makes the dimension of the linear moduli space $d(n+m-d) - |\mc{R}|$. 

While the face lattice of a polytope is a larger set of objects than the facet-vertex incidence relation $\mc{R}$, the following lemma shows that the face lattice is the maxbiclique lattice of $\mc{R}$, so these objects carry the same information.  
This is proved in \cite{introlat} as a special case of the Dedekind-MacNeille completion. 

\begin{lem}\label{lem:mbc} Any lattice where all flags are finite is isomorphic to the maxbiclique lattice of the comparability relation between its meet and join irreducibles.  
\end{lem}

For polytopes, this means that each face of a polytope can be uniquely determined from the vertices it contains and the facets it is contained in, and a set of facets and vertices are those of a face if and only if the following 2 conditions hold; (biclique) all the vertices of the face are contained in all its facets, and (maximal) this first condition fails if any other vertex or facet is added to those of the face.

The proof of Theorem \ref{tem:rmcp} will use ideas similar to what D\'iaz used in \cite{diaz}, but are presented in a more combinatorial rather than geometric manner.  The point of this is to make explicit the purely combinatorial conditions that are needed, rather than using combinatorial polytopes, which are geometrically defined objects. 
The proof works by constructing a polytope and showing that there is an order preserving injection from the maxbiclique lattice of the given relation to the face lattice of this polytope.  The following lemma shows that such an injection is actually a bijection, so exhibiting such an injection is enough to show that the polytope realizes the given relation. 

\begin{lem}\label{apinj} An order monomorphism between flag connected diamond lattices of the same finite rank is an isomorphism.
\end{lem}

\begin{proof}
Suppose the lemma fails, then there are flag connected diamond lattices $\mathcal{P}$ and $\mathcal{Q}$ of rank $d\!+\!1$ with a monomorphism $\varphi: \mc{P} \to \mc{Q}$ such that there is some element $f\in\mathcal{Q}$ that is not in $\varphi(\mc{P})$.  Without loss of generality we assume that $\mc{P} \subset \mc{Q}$ and the monomorphism is the identity, otherwise just replace $\mathcal{P}$ with its image.  Consider now the flag graphs $G,H$ of $\mc{P},\mc{Q}$ respectively.  Every flag of $\mc{P}$ is a flag of $\mathcal{Q}$, and two flags of $\mc{P}$ differ by one element as flags of $\mc{P}$ if and only if they do so as flags of $\mathcal{Q}$, so $G$ is an induced subgraph of $H$.  The lattice $\mc{P}$ is non-empty, and $f$ must belong to some flag of $\mathcal{Q}$, so $\emptyset \neq G \varsubsetneq H$.  
The diamond condition is equivalent to each flag having 1 neighbor in the flag graph for each element except $\top$ and $\bot$, and no other neighbors.
This means $G$ is a $d$-regular proper induced subgraph of a connected $d$-regular graph, $H$, which is impossible. 
To see why, consider a path of $H$ from a node that is in $G$ to a node that is not in $G$.  The last node of this path that is in $G$ must have $d$ neighbors total, of which at least 1 neighbor must not be in $G$ and $d$ neighbors must be in $G$. 
\end{proof}

To show that the conditions for realizability given in Theorem \ref{tem:rmcp} are sufficient, 
we construct a $d$-polytope $P$  
from a decomposition of a rank $d$ filled 1-incidence matrix $M=H^* W$ of a relation $\mc{R}$ satisfying the specified combinatorial conditions into a $n \times k$ matrix $H^*$ and a $k \times m$ rank $d$ matrix $W$. 
Specifically, $P$'s vertices are the columns of $W$, and if $k=d$ its covertices are the columns of $H$. 
We then see that the facet-vertex incidence of the constructed polytope $P$ is $\mc{R}$. 
To make this construction more explicit, we use the compact singular value decomposition to get $H$, $W$ and their pseudoinverses $H^+$, $W^+$. 

If we start with a polytope, find its facet-vertex matrix $M$, and do the above construction, we get back a linear copy of the polytope we started with. 
This is because, 
if the columns of a matrix $X$ are the vertices of a polytope $P$ and $W$ is any matrix with the same null space as $X$, then  $X=AW$ factors though $W$ by a matrix $A=XW^+$ that does not change the rank, so 
the columns of $W$ are the vertices of some linear copy of $P$. 
We can do the above construction by finding a bases for ${\rm null}(M)^\bot = {\rm null}(X)^\bot$, which are the rows of $W$, to get a linear copy of $P$ in $\mb{R}^d$.  Alternatively, we get the following from the trivial decomposition $M = I^*M$. 
\begin{cor}
The columns of the facet-vertex matrix $M$ of a polytope $P$ are the vertices of a linear copy of $P$.
\end{cor} 

The is shown in the proof of Theorem \ref{tem:rmcp}. 
We would, however, like to start with just a matrix $M$ instead of a polytope.

\begin{proof}[Proof of Theorem \ref{tem:rmcp}]
We have the `only if' direction of the first part of the theorem immediately. 
To show that any relation $\mc{R}$ satisfying the combinatorial conditions with a rank $d$ filled 1-incidence matrix $M$ can be realized as the facet-vertex incidence of a polytope, 
let $\mc{P}$ be the maxbiclique lattice of $\mc{R}$, which by hypothesis is rank $d+1$, flag connected, and satisfies the diamond condition, let $H$ and $W$ be matrices such that $W$ has rank $d$ and $M=H^*W$, 
let $\{h_i\}_{i\in[n]}$ and and $\{w_j\}_{j\in[m]}$ be the columns of $H$ and $W$ respectively, 
and let $P=\text{conv}(\{w_j\}_{j\in[m]})$.  

We will show that $P$ is a realization of $\mc{R}$.  
Since $W$ has rank $d$, $P$ has dimension at most $d$.
Furthermore, $w_j$ is in the hyperplane $h_i^{=1}:=\{x: \left<h_i,x \right>=1\}$ for $(i,j)\in\mc{R}$, but is in the open half space $h_i^{<1}:={\{x: \left<h_i,x\right><1\}}$ for $(i,j)\nin\mc{R}$. 
To be more explicit we could let $H=I$ and $W=M$ for example, in which case $h_i^{=1}$ becomes $[x]_i = 1$ and $h_j^{<1}$ becomes $[x]_i < 1$.

We will construct a map from $\mc{P}$ to the face lattice of $P$ and show that it is an isomorphism. 
For $a = (I_a,J_a) \in \mc{P}$, let $\varphi(a)=\bigcap_{i\in I_a}h_i^{=1} \cap P$ be the face of $P$ we get by intersecting it with the hyperplanes corresponding to covertices of $a$.  This must be a face of $P$ since $h_i^{=1}$ are all supporting hyperplanes of $P$. 

First we will see that $\varphi$ preserves order.  That is, we will show the very strong condition that $\varphi(a) \subseteq \varphi(b)$ if and only if $a\leq b$.  Suppose $a\leq b$, then  $I_a\subseteq I_b$ and $\bigcap_{i\in I_a}h_i^{=1}\subseteq \bigcap_{i\in I_b}h_i^{=1}$ so $\varphi(a) \subseteq \varphi(b)$.  For the other direction suppose $a \nleq b$, then there is some $j\in J_a$ but $j\nin J_b$, so $w_i\in \varphi(a)$ but $w_j\nin \varphi(b)$ and $\varphi(a)\nsubseteq \varphi(b)$.  Thus order is maintained.

We will also see that $\varphi$ is an injection, and therefor a monomorphism.  To see this consider a pair of faces $a, b$ of $\mc{P}$ that map to the same face $\varphi(a)=\varphi(b)=f$ of $P$.  With this, $w_j\in f \subset h_i^{=1}$ for any $i\in I_a\cup I_b$ and $j\in J_a\cup J_b$, so $[M]_{ij}=1$ and 
$(i,j) \in \mc{R}$. 
This means $a=(I_a,J_a)$ and $b=(I_b,J_b)$ are maxbicliques that are subsets of the same biclique $c=(I_a\cup I_b, J_a\cup J_b)$, so they must be equal $a=b=c$, which makes $\varphi$ is a monomorphism.  
 
This induces an injection from a flag $\bot,a_1,\dots,a_d,\top \in \mc{P}$ to a totally ordered set of $P$'s faces $\varphi(\bot)=\emptyset,\varphi(a_1),\dots,\varphi(a_d),\varphi(\top)=P$, which must be of the same size or less.  Since a larger set cannot be injected into a smaller one, these flags must be the same size, so $P$ must have dimension $d$. 
Now $\varphi$ is a monomorphism between flag connected diamond lattices of the same rank, and by Lemma \ref{apinj} is therefore an isomorphism.  Thus, $P$ is a realization of $\mc{R}$ and the `if' direction of the first part of the theorem holds.


For the second part of the theorem, we consider realizations in $\mb{R}^d$.  For this, let $M = U\Sigma V^*$ be the compact singular value decomposition of $M$ and let $H = \Sigma^{\nicefrac{1}{2}} U^*$ and $W = \Sigma^{\nicefrac{1}{2}} V^*$. 
Recall that any positive symmetric definite matrix is the square of another unique positive symmetric definite matrix and $\Sigma$ is a positive diagonal matrix, so this is well defined. 
From what we have just seen, $M$ is the facet-vertex matrix of the polytope $P$ with vertices given by the columns of $W$. 
Suppose we have some other decomposition $M=H_0^*W_0$ of the same size.  The rows of both $W$ and $W_0$ give a basis of the same space, ${\rm null}(M)^\bot$, so there is an invertible linear transformation $A$ between them $W_0=AW$.  For $W$ as given this is $A=W_0V\Sigma^{-\nicefrac{1}{2}}$.  Note that $W^+ = V\Sigma^{-\nicefrac{1}{2}}$ is the pseudo-inverse of $W$.  We can see this formula more clearly by considering a completion of $V$ to a full orthogonal matrix $Q$. 
\[ AW = W_0V\Sigma^{-\nicefrac{1}{2}}\Sigma^{\nicefrac{1}{2}}V^* 
 = W_0\: [ V \: \mb{0} ] { V^* \brack \mb{0} }
 = W_0\: [ V \: \mb{0} ] Q^*
= W_0QQ^* = W_0 \]
Note we are able to replace $\mathbb{0}$ columns with the extra columns of $Q$ 
since every row of $W_0$ is orthogonal to the extra columns of $Q$, so these columns vanish anyway. 
Similarly we have $A^{-*} = H_0U\Sigma^{-\nicefrac{1}{2}}$ and $H_0=A^{-*}H$, and for any $A$ nonsingular we get such a decomposition $M=(A^{-*}H)^*AW$.  This means that if $P_0 \subset \mb{R}^d$ is any centered type $\mc{R}$ polytope with facet-vertex matrix $M$, then $P_0 = A P$ is a linear copy of $P$, since the matrices formed by vertices $W_0$ and covertices $H_0$ of $P_0$ give another decomposition of $M$.  Also, applying an invertible linear transformation to $P$ will not change the facet-vertex matrix of the resulting polytope, so $M$ parameterizes centered realizations of $\mc{R}$ modulo linear transformations. 
\end{proof}

Here is the main theorem for cones.

\begin{tem}\label{tem:rmcc}
If the maxbiclique lattice of $\mc{R}$ is 
flag connected, rank $d+1$, and diamond, then $\mc{R}$ has 
a rank $d$ filled 1-incidence matrix if and only if  
$\mc{R}$ has a rank $d+1$ filled 0-incidence matrix $N$.

Moreover, 
for such $\mc{R}$, $N$ is the facet-ray matrix of a type $\mc{R}$ cone if and only if $N$ is a rank $d+1$ filled 0-incidence matrix $N$, and such matrices $N$ modulo positive scalings of its rows and columns  parameterize the projective moduli space of type $\mc{R}$ polytopes. 
\end{tem}

\begin{proof}
For the `only if' direction, Theorem \ref{tem:rmcp} gives us a polytope realizing $\mc{R}$ and a facet-ray matrix of the cone over this polytope in homogeneous coordinates gives us a rank $d+1$ 0-incidence matrix for $\mc{R}$. 
For the rest of the proof, we exclude the case where $\mc{R}$ is a singleton or empty. 
We could show the other direction by going through almost the same argument as Theorem \ref{tem:rmcp}, except for cones rather than polytopes, 
but we prefer to show that we can use a rank $d+1$ filled 0-incidence matrix $N$ for an appropriate relation $\mc{R}$ to find a rank $d$ filled 1-incidence matrix $M$ for $\mc{R}$.  We first note that $N+\mb{1}$ is a filled 1-incidence matrix that and can differ from $N$ in rank by at most 1.  If the rank of $N + \mb{1}$ were 1 less than that of $N$ we would be done, but that might not be the case.  Instead we will find full diagonal matrices $D_1, D_2$ such that $M = D_1ND_2+\mb{1}$ has rank $d$.  We start by letting $N=U\Sigma V^*$ be the compact singular value decomposition and $U_0$ and $V_0$ be respective completions of $U$ and $V$ to orthogonal matrices.

There is some positive vector $y\in \mb{R}^{d+1}_{>0}$ in the positive orthant such that the entries of $Vy$ are all non-zero. 
To see why, suppose no such vector exists.  That is every vector in $V\mb{R}^{d+1}_{>0}$ has some zero entry.  This means that every vector of the positive orthant is orthogonal to some row of $V$, but the orthogonal compliment of any non-zero vector has measure 0, so
the assumption can only hold if 
$V$ has a row of all zeros.  
In that case $N$ would have a column of all zeros, which means some meet irreducible of $\mc{R}$ is incident to all join irreducibles, so this must be the only meet irreducible, but $\mc{R}$ is not a singleton so the diamond condition implies there is more than 1 meet irreducible.
Thus we can find a vector $y$ with all positive entries such that $Vy$ has all non-zero entries, and similarly we can find such a vector $x$ for $U$.  Furthermore, we can normalize these so that $\left<x,\Sigma^{-1}y\right>=1$.

With this in mind we let  $D_1 = {\rm diag}(-Ux)^{-1}$ and $D_2 = {\rm diag}(Vy)^{-1}$, and we preform a change of bases to make the rank of the resulting matrix more apparent.
\begin{align*}
 (U_0^*D_1^{-1})M(D_2^{-1}V_0)
& = U_0^*NV_0 + U_0^*D_1^{-1}\mb{1}D_2^{-1}V_0 \\
& = U_0^*U\Sigma V^*V_0 + U_0^*D_1^{-1}\mathbb{1}\mathbb{1}^*D_2^{-1}V_0 \\
& = U_0^*U(\Sigma - xy^*)V^*V_0 \\
& = \left[\begin{array}{cc}
     (\Sigma - x y^*) & \mathbb{0} \\
     \mathbb{0} & \mathbb{0}
    \end{array}\right]    
\end{align*}
Under this change of bases we see $e_i$ is sent to $\mathbb{0}$ for $i>d+1$, so the rank is at most $d+1$.  We observe that it is in fact $d$, by showing the range is orthogonal to ${\Sigma^{-1}y \brack \mathbb{0}}$.  For this we consider only the first $d+1$ coordinates.
\begin{align*}
 \left< (\Sigma - x y^* )e_i, \Sigma^{-1} y \right> 
 & = \left< \Sigma e_i - (y)_ix , \Sigma^{-1} y \right> \\
 & = [y]_i - [y]_i = 0
\end{align*}
Thus $M$ has rank $d$ and the `if' direction holds.

For the second part of the theorem, recall that two polytopes $P$ and $Q$ are projectively equivalent if and only if the cones over these polytopes in homogeneous coordinates are linearly equivalent.  That is for some $A$ invertible, $A(\mb{R}_{\geq 0}{P \brack 1}) = \mb{R}_{\geq 0}{Q \brack 1}$.  
As we have seen, full rank matrices $H$ and $W$ of a decomposition $N=H^*W$ are determined up to linear transformation by $N$, so $N$ determines a set of generators and cogenerators of $C$ up to linear transformation. 
Also, a set of generators and cogenerators determine the same cone if and only if they are equivalent up to  positive scaling of each, which amounts to scaling the rows and columns of $N$. 
In this way $N$ modulo positive scaling of rows and columns uniquely determines a polytope up to projective transformation. 
Note that the diagonal matrices $D_i$ above scale the generators and cogenerators of a cone to be the vertices of a polytope and the vertices of a dual polytope, which come from the cone and its dual by intersecting each with a hyperplane, and the choice of $x$ and $y$ amount to choosing these hyperplanes.  
\end{proof}

The flag graph of a polytope may be much larger than its face lattice, but is highly structured.  The following theorem shows that we do not have to consider the full flag graph when deciding realizability.

\begin{tem}\label{tem:connect}
The flag connected condition in Theorems \ref{tem:rmcp} and \ref{tem:rmcc} can be replaced with the following condition.  
For every element $a$ of the maxbiclique lattice with rank $k$, the graph having rank $k-1$ elements $a_i \prec a$ as vertices and rank $k-2$ elements $a_i \wedge a_j $ as edges is connected. 
\end{tem}

\begin{proof}
For the `only if' direction, this is the graph of the facets and ridges of each face of the polytope, which is the 1-skeleton of its order dual, and the 1-skeleton of a polytope is always connected. 

For the `if' direction, consider two flags $\mc{F} = [\bot, f_0, \dots, f_{d-1}, \top]$, $\mc{G} = [\bot, g_0, \dots, g_{d-1}, \top]$ of a diamond lattice satisfying the new condition. 
We will build a path from $\mc{F}$ to $\mc{G}$ in the flag graph inductively from facets down to vertices represented by the following object at each step. 
For $c \in [0,d-1]$ we claim there is a sequence of faces 
$H_c = [h_{c,1}, \dots, h_{c,t_c}]$ such that for every $k \in [c,d-1]$ the subsequence of rank $k$ faces begins with $f_k$, ends with $g_k$, and for any two consecutive faces $h_{c,i}$, $h_{c,j}$ in the subsequence with $i<j$ the following holds. 
$h_{c,i} \vee h_{c,j}$ is last face of rank $k+1$ in the sequence before $h_{c,j}$ if $k < d-1$ and is $\top$ if $k=d-1$, and $h_{c,i} \wedge h_{c,j}$ is the last face of rank $k-1$ in the sequence before $h_{c,j}$ if $k > c$ and is some face of rank $c-1$ if $k=c$. 

Before seeing why the claim holds, we will see why it implies the lattice is flag connected.  Consider the sequence $H_0$, disregarding the faces of $\mc{F}$ in this sequence.  Starting from $\mc{H}_0 = \mc{F}$ we get a sequence of flags $\mc{H}_i = \mc{H}_{i-1} \cup h_{0,i} \setminus h_{0,j}$ where $h_{0,j}$ is the face of $\mc{H}_{i-1}$ with the same rank as $h_{0,i}$.  Note that $h_{0,j}$ is also the last face in $H_0$ before $h_{0,i}$ of the same rank.  By the claim, $h_{0,i} \vee h_{0,j}$ is the face of $\mc{H}_i$ with rank $k+1$ and $h_{0,i} \wedge h_{0,j}$ is the face of $\mc{H}_i$ with rank $k-1$, so this is a path in the flag graph from $\mc{F}$ to $\mc{G}$. 

Now we will see why the claim holds. The case $c=d-1$ is exactly the new condition for the face $\top$.  Assume by induction that the claim hold for the case $c$ and let $h_{c,x}$, $h_{c,y}$, $h_{c,z}$ be consecutive faces in the subsequence of faces with rank $c$.  By the new condition for $h_{c,y}$ there is a sequence $ E_y = [e_{y,1}, \dots e_{y,s_y}]$ such that $e_{y,1} = h_{c,x} \wedge h_{c,y}$, $e_{y,s_y} = h_{c,y} \wedge h_{c,z}$, $e_{y,i} \vee e_{y,i+1} = h_{c,y}$, and $e_{y,i} \wedge e_{y,i+1}$ has rank $c-1$.  Inserting $E_y$ after each face $h_{c,y}$ of rank $c$ in the sequence $H_c$ then gives the desired sequence $H_{c-1}$.  By induction we get a sequence $H_0$ that determines a path in the flag graph between $\mc{F}$ and $\mc{G}$ in the flag graph, and since this hold for any pair of flags, the flag graph is connected. 
\end{proof}

\section{Definitions}\label{sec:def2}

So far we have considered realizability in Euclidean space in terms of the facet-vertex or facet-ray matrices, now we will give definitions for finding realizations in other geometric spaces in terms of the Gramian matrix of a polytope. 
A geometric space is a level set of a bilinear form  $\phi$,  
which is a function on two vectors of the form $\phi(x,y)=x^*\Phi y$ for some matrix $\Phi$. 
If $\Phi = I$ is the identity matrix, for example, then $\phi(x,x)=\left<x,x\right>$ is the standard inner product and $\phi(x) = \|x\|^2$ is the standard norm squared.  
A $d$ dimensional affine space $\mb{A}^d$ may be represented in homogeneous coordinates as a non-zero level set of a linear functional on $\mb{R}^{d+1}$, and polyhedra in $\mb{A}^d$ are given by the intersection of $\mb{A}^d$ with a cone in $\mb{R}^{d+1}$, not necessarily pointed. 
Similarly, a geometric space $\mb{X}$ is given by a connected component of a non-zero level set of $\phi$, 
and polyhedra in $\mb{X}$ are given by the intersection of $\mb{X}$ with a cone. 

We will mostly be concerned with realizing cones. 
Recall the generators and cogenerators of a given cone are not unique.  However, we can use $\phi$ to define a unique set of cogenerators for certain cones as follows.  A hyperplane is called lightlike when $\phi$ vanishes on the hyperplane away from the origin.  A half-space bounded by a hyperplane that is not lightlike has the form $h^{\phi, \leq 0} = \{ x: \phi(h,x) \leq 0 \}$ for a unique {\bf outward normal} vector $h$ such that $|\phi(h)| = 1$.  Observe that if $\phi(h,h) = 0$ then $h$ is the the hyperplane $h^{\phi, = 0} = (\Phi h)^{=0}$. 
For a cone $C$ such that no facet supporting hyperplane is lightlike, the outward normals $h_i$ of the facet supporting half-spaces are cogenerators and the {\bf Gramian} $G$ of $C$ with respect to $\phi$ has entries $[G]_{i,j} = \phi(h_i,h_j)$.  

We will refer to elements of a lattice as the corresponding object in the face lattice of a polytope. 
A {\bf cycle} is a sequence of facets $F_{i_1},\cdots,F_{i_d}$ such that $\bigwedge_{k=1}^d F_{i_k}$ is a vertex and, a {\bf super cycle} is a cycle with an additional facet such that $\bigwedge_{k=1}^{d+1} F_{i_k} = \bot$. 
Note that while D\'iaz called these maximal cycles, we will call these super cycles instead to emphasize that they are not cycles. 
A cycle or super cycle induces a flag with faces $\bigwedge_{k=1}^t F_{i_k}$ for $1\leq t \leq d$, 
$\top$, $\bot$. Note the flag graph of a combinatorial polytope $\mc{P}$ is always bipartite, with 
bipartition given for each flag by the orientation of the outward normal vectors of the facets of a polytope  realizing $\mc{P}$ in a cycle inducing that flag. 
With this in mind, we we will require the flag graph to be bipartite and say two super cycles have the same {\bf orientation} when the induced flags are in the same bipart.  
We denote the minor of a matrix $G$ with rows $i_1\cdots i_k$ and columns $j_1\cdots j_l$ by $[G]{i_1\cdots i_k \atop j_1\cdots j_l}$.  
The {\bf signature} of a matrix, or bilinear form, is the pair $(p,n)$ where $p$ and $n$ are the total ranks of eigenspaces with positive and negative eigenvalues respectively.  
For $\Phi$ with signature $(d+1,0)$, such as with $\Phi = I$, a level set $\mb{X}$ has spherical geometry, and for signature $(d,1)$, $\mb{X}$ has hyperbolic geometry. 
A linear transformation $Q$ is an {\bf orthogonal} transformation of $\phi$ if it preserves $\phi$, $Q^*\Phi Q = \Phi$.  In the case $\Phi = I$, these are the usual orthogonal transformations.

We have defined everything needed to understand the statement of Theorem \ref{tem:gram}, but its proof uses exterior algebra, which we briefly review now. 
The exterior power $\bigwedge^m X^*$ is the space of $m$-ary multilinear functions on $X$ that vanish on linearly dependent vectors, and the dual 
$\bigwedge^m X$ consists of equivalence classes of $m$-tuples, denoted $x_1\wedge \cdots \wedge x_m$, where two $m$-tuples are equivalent when they evaluate to the same result by every function in $\bigwedge^m X^*$, so in particular 
$ y_1 \wedge \cdots \wedge y_m = 0 $ 
when $\{y_1, \dots, y_m\}$ is linearly dependent.  
For any real vector space, $\bigwedge^0 X:= \mathbb{R}$ by convention. 
Transpositions on an $m$-tuple alternate sign, 
\[(\cdots \wedge x_i \wedge \cdots \wedge x_j \wedge \cdots) = - (\cdots \wedge x_j \wedge \cdots \wedge x_i \wedge \cdots).\]   
We also treat $(\wedge):\bigwedge^n X \times \bigwedge^m X \to \bigwedge^{n+m} X $ as a function acting on exterior powers by concatenation 
\[(x_1 \wedge \cdots \wedge x_n) \wedge (x_{n+1} \wedge \cdots \wedge x_{n+m}) = x_1 \wedge \cdots \wedge x_{n+m}. \]   
Given a bilinear form $\phi$ on $X$ we define a bilinear form on $\bigwedge^m X$ by 
\[(\bigwedge{\!}^m \phi)(y_1 \wedge \cdots \wedge y_n, x_1 \wedge \cdots \wedge x_n):= {\rm det}(A).\] 
where $[A]_{i,j} = \phi(y_i,x_j)$. 
Recall that real vector spaces of the same rank are isomorphic, and notice the spaces $\bigwedge^m X$ and $\bigwedge^{r-m} X$ have the same rank ${r \choose m} = {r \choose r-m}$.  The {\bf Hodge star operator} $\star: \bigwedge^m X \to \bigwedge^{r-m} X$ is the canonical isomorphism between these defined with respect to $\phi$ as follows. 
In the case $m=r$, $\star$ is the unique linear map such that $\star(y_1 \wedge \cdots \wedge y_r):=1$ for any positively oriented basis $\{y_i\}_{i\in [r]}$ of $X$ that is orthonormal with respect to $\phi$, $|\phi (y_i)|=1$, $\phi(y_i,y_j) = 0$ for $i\neq j$.  
Note in the case where $\phi$ is the standard inner product, $\star$ on $\bigwedge^r X$ is the determinant. 
This gives two ways to define linear functionals on $v\in\bigwedge{\!}^{r-m} X$; either use $u\in \bigwedge{\!}^{m} X$ to augment $v = (x_{m+1} \wedge \cdots \wedge x_r)$ to $u\wedge v = (x_1 \wedge \cdots \wedge x_r) \in \bigwedge{\!}^{r} X$ and take $\star$ of the result, or use $u\in \bigwedge{\!}^{r-m} X$ and take the bilinear form of the pair $(\bigwedge{\!}^{r-m} \phi)(u,v)$.  The Hodge star operator is the correspondence between these two ways of representing linear functionals. 
That is, 
$\star := \tau^{-1} \circ \psi$ where
\begin{align*}
\psi: \bigwedge{\!}^m X \ \to \left(\bigwedge{\!}^{r-m} X \right)^*,& \hspace{12pt}  \psi(u) := v \mapsto \star(u\wedge v) \\
\tau: \bigwedge{\!}^{r-m} X \to \left(\bigwedge{\!}^{r-m} X \right)^*,& \hspace{13pt} \tau(x) := v \mapsto (\bigwedge{\!}^{r-m} \phi )(u, v) 
\end{align*}
In the case of $m=1$, $\bigwedge{\!}^1 X = X$, $\bigwedge{\!}^1 \phi = \phi$ this becomes   
\[ \phi (\star(x_1 \wedge \cdots \wedge x_{r-1}), x_r) = \star(x_1 \wedge \cdots \wedge x_{r-1} \wedge x_r ).  \] 
Note that in $\mb{R}^3$ with $\Phi = I$ this becomes 
$\left< x_1 \times x_2, x_3 \right> = \det([x_1 x_2 x_3])$.

\section{The Gramian}\label{sec:gram}

This section will give realizability conditions in term of the Gramian of a polytope. 

\begin{tem}\label{tem:gram}
Given a $n \times m$ relation $\mc{R}$ and symmetric bilinear form $\phi$, 
there is a type $\mc{R}$ polytopal cone with Gramian $G$ if and only if 
the maxbiclique lattice of $\mc{R}$ is a rank $d+1$ diamond lattice with connected bipartite flag graph, and $G$ is a $n \times n$ symmetric matrix with diagonals $\pm 1$ and the same signature as $\phi$
that satisfies the following: 
\begin{enumerate}[label=\arabic*]
 \item\label{con:vertrank} For every vertex of $\mc{R}$ and all facets $F_{i_1}\cdots F_{i_n}$ incident to it, ${\rm rank}([G]{i_1 \cdots i_n \atop i_1 \cdots i_n}) = d$.
 \item\label{con:supcyc} For every pair of super cycles  $F_{i_1},\cdots, F_{i_{d+1}}$ and $F_{j_1},\cdots, F_{j_{d+1}}$  with the same orientation, $\det([G]{i_1\cdots i_{d+1} \atop j_1\cdots j_{d+1}}) {\rm det}(\phi) > 0$.
\end{enumerate}
Moreover, $G$ parameterizes the realization space of type $\mc{R}$ cones modulo orthogonal transformations of $\phi$. 
\end{tem}

D\'{i}az proved a theorem very similar to this giving conditions for finding a cone with Gramian $G$, but with $\mc{R}$ replaced by a combinatorial polytope, and  
the proof of Theorem \ref{tem:gram} will follow that given by D\'iaz \cite{diaz}, except 
Theorem \ref{tem:rmcc} will be used in place of the corresponding result for combinatorial polytopes.  
Theorem $\ref{tem:rmcc}$ makes use of both the generators and cogenerators of a cone, but Theorem \ref{tem:gram} only makes use of the cogenerators.  To deal with this, we use the Hodge star operator $\star$ to construct generators of a cone from its cogenerators in a way that is analogous to finding the generators of a polygonal cone in $\mb{R}^3$ with $\Phi=I$ by taking the cross product of outward normals of neighboring facets.  




\begin{proof}[Proof of Theorem \ref{tem:gram}]
Since we are only weakening one side of a biconditional statement in the theorem given by D\'iaz, we only have to consider the argument showing one direction.  That is, we will show that if a relation and a matrix satisfy these conditions, then they can be realized by such a cone.  We only consider $d\geq 1$, since the theorem is trivial otherwise.  We can represent $\phi$ by a real symmetric matrix $\Phi$, and 
since $G$ and $\Phi$ are real symmetric matrices with the same signature, there exists a $n \times (d\!+\!1)$ matrix $H$ such that $G=H^*\Phi H$.  We denote columns of $H$ by $h_j$, and we claim $C:=\bigcap_j h_j^{\phi,\leq 0}$ 
is a polytopal cone of type $\mc{P}$.  Following in the footsteps of \cite{diaz}, we show this by using the Hodge star operator to find generators of $C$, which we collect in a matrix $W$, and show that $N=(\Phi H)^* W = H^*\Phi W$ is a rank $d\!+\!1$ filled 0-incidence matrix of $\mc{P}$.  

With this in mind, choose a cycle $F_{i_{j,1}}, \cdots, F_{i_{j,d}}$ incident to each vertex $v_j$ of $\mc{P}$ such that these all have the same orientation, and let 
\[ w_j = \star(h_{i_{j,1}} \wedge \cdots \wedge h_{i_{j,d}}).\]   
Since $\mc{P}$ is a diamond lattice with rank at least 2, we can choose another facet $F_{i_{j,d\!+\!1}}$ that is not incident to $v_j$ giving a super cycle.  By Condition \ref{con:supcyc} for all $j$ we have 
${\rm det}([G]{i_{j,1}\cdots i_{j,d\!+\!1} \atop i_{j,1}\cdots i_{j,d\!+\!1}}) \neq 0$, so the vectors $\{h_{i_{j,k}}\}_{k\in [d+1]}$ are linearly independent and $h_{i_{j,1}} \wedge \cdots \wedge h_{i_{j,d}}$ is not the origin of $\bigwedge^d \mb{R}^{d+1}$. 
Therefor, $w_j \neq \mathbb{0}$, since $\star$ is an isomorphism.  We let $W$ have columns $w_j$ and $N=H^*\Phi W$. 

We now show that $N$ is a 0-incidence for the relation between the irreducibles of $\mc{P}$.  
For $v_j < F_i$ incident, by Condition \ref{con:vertrank} we have 
\[[N]_{i,j} = \phi (h_i,w_j) = \phi (h_i,\star(h_{i_{j,1}} \wedge \cdots \wedge h_{i_{j,d}})) = \star(h_i \wedge h_{i_{j,1}} \wedge \cdots \wedge h_{i_{j,d}}) = 0, \] 
since $\{h_i, h_{i_{j,1}}, \dots, h_{i_{j,d}}\}$ has span at most $d$ and is therefor linearly dependent.
Alternatively consider a pair $v_j \not< F_i$ that are not incident.  As we have seen, this means 
$\{h_{i_{j,1}}, \cdots, h_{i_{j,d}}, h_i\}$ are linearly independent and therefor a basis of $\mb{R}^{d+1}$. 
We have already that $\phi (h_{i_{j,k}}, w_j)=0$ for $k \in [d]$, and since $\phi$ has full rank we cannot have $\phi (x, w_j)=0$ for every vector $x$ of a basis, so we must have $\phi (h_i,w_j) \neq 0$.  Thus $[N]_{i,j}=0$ if and only if $F_i$ and $v_j$ are incident in $\mc{P}$.

We would like to show $N$ is non-positive.  We will actually show $N$ is either that or nonnegative, which can be fixed by redefining $W$ to be $-W$, so this is enough.  We do this by showing all nonzero entries have the same sign. 
Recall that we chose cycles that have the same orientation, but did not specify which orientation, so we should not expect to have chosen correctly. 
Let $(i,j)$ and $(\imath,\jmath)$ both be indices of a pair of irreducible of $\mc{P}$ that are not incident.  By Condition \ref{con:supcyc} we have the following. 
\[ \phi (h_i,w_j)\phi (h_\imath,w_\jmath)
 = {\rm det}\left([G]^{ i_{j,1} \cdots i_{j,d}\: i}_{ i_{\jmath,1} \cdots i_{\jmath,d}\: \imath }\right)
  {\rm sign}(\phi ) >0 \]

This gives us that $N=(\Phi H)^* W$ is a rank $d\!+\!1$ filled 0-incidence matrix of $\mc{P}$, adjusting $W$'s sign if needed. By Theorem \ref{tem:rmcc}, $C$ is a type $\mc{P}$ cone and the columns of $\Phi H$ are cogenerators of $C$.  Thus, $C$ has outward normals $\{h_i\}_{i\in [n]}$ and Gramian $G$ with respect to $\phi$. 
\end{proof}

We now state versions the above theorem for spherical and hyperbolic polytopes.  D\'iaz gave proofs of these for combinatorial polytopes, which work just as well in this context, so the proofs are not included here.  Let $\mb{S}^d = \{x : \|x\| = 1\}$ denote the $d$-dimensional unit sphere in $\mb{R}^{d+1}$. A spherical $d$-polytope is the intersection of a pointed polytopal cone in $\mb{R}^{d+1}$ with $\mb{S}^d$, and the Gramian of this polytope is the Gramian of the cone. 

\begin{tem}
Given a $n \times m$ relation $\mc{R}$, 
there is a type $\mc{R}$ spherical polytope with Gramian $G$ if and only if 
the maxbiclique lattice of $\mc{R}$ is a rank $d+1$ diamond lattice with connected bipartite flag graph, and $G$ is a $n \times n$  rank $d+1$ symmetric positive semi-definite matrix with all diagonals $1$ 
that satisfies the following: 
\begin{enumerate}[label=\arabic*]
 \item For every vertex of $\mc{R}$ and all facets $F_{i_1}\cdots F_{i_n}$ incident to it, ${\rm rank}([G]{i_1 \cdots i_n \atop i_1 \cdots i_n}) = d$.
 \item\label{con:diaz'} For every pair of super cycles  $F_{i_1},\cdots, F_{i_{d+1}}$ and $F_{j_1},\cdots, F_{j_{d+1}}$  with the same orientation, $\det([G]{i_1\cdots i_{d+1} \atop j_1\cdots j_{d+1}}) > 0$. 
\end{enumerate}
Moreover, $G$ parameterizes the isometric moduli space of type $\mc{R}$ spherical polytopes. 
\end{tem}

Let $\Phi = {\rm diag}(1,\dots,1,-1)$ so $\phi(x) = [x]_1^2 + \dots + [x]_d^2 -[x]_{d+1}^2$, and let $\mb{H}^d = \{ x: \phi(x) = -1, [x]_{d+1} > 0\}$  denote the $d$-dimensional upper hyperbola in $\mb{R}^{d+1}$. 
A finite-volume hyperbolic $d$-polytope $P=\mb{H}^d \cap C$ is the intersection of $\mb{H}^d$ with a pointed polytopal cone 
\[ C \subset \{x : \phi(x) \leq 0, [x]_{d+1} \geq 0 \}.\] 
Again the Gramian of $P$ is the Gramian of $C$.  When a ray of $C$ does not intersect $\mb{H}^d$, we say the corresponding vertex, which appear as any point in $\mb{H}^d$, is an ideal vertex of $P$.  An isometry of $\mb{H}^d$ is an orthogonal transformation of $\phi$. 

\begin{tem}
Given a $n \times m$ relation $\mc{R}$ and vertices $\mc{W} \subset [m]$ of $\mc{R}$, 
there is a finite-volume type $\mc{R}$ hyperbolic polytope with all vertices except those of $\mc{W}$ in $\mb{H}^d$ and Gramian $G$ if and only if 
the maxbiclique lattice $\mc{P}$ of $\mc{R}$ is a rank $d+1$ diamond lattice with connected bipartite flag graph, and $G$ is a $n \times n$ symmetric matrix with all diagonals $1$ 
that satisfies the following: 
\begin{enumerate}[label=\arabic*]
 \item For every vertex of $\mc{R}$ and all facets $F_{i_1}\cdots F_{i_n}$ incident to it, ${\rm rank}([G]{i_1 \cdots i_n \atop i_1 \cdots i_n}) = d$.
 \item\label{con:diaz'} For every pair of super cycles $F_{i_1},\cdots, F_{i_{d+1}}$ and $F_{j_1},\cdots, F_{j_{d+1}}$  with the same orientation, $\det([G]{i_1\cdots i_{d+1} \atop j_1\cdots j_{d+1}}) < 0$. 
 \item For $2\leq s \leq d$ and every truncated cycle $F_{i_1},\cdots, F_{i_s}$ incident to a $(d\!-\!s)$-face $f\in \mc{P}$, if $f \in \mc{W}$ then $\det(G_{i_1 \cdots i_s}) = 0$ otherwise $\det(G_{i_1 \cdots i_s}) > 0$
 \item\label{con:diaz'} For every pair of super cycles $F_{i_1},\cdots, F_{i_{d+1}}$ and $F_{j_1},\cdots, F_{j_{d+1}}$ with the same orientation that are incident to a different vertex, $\det([G]{i_1\cdots i_{d+1} \atop j_1\cdots j_{d+1}}) > 0$. 
\end{enumerate}
Moreover, $G$ parameterizes the isometric moduli space of finite-volume type $\mc{R}$ hyperbolic polytopes with ideal vertices $\mc{W}$. 
\end{tem}

\section{Concluding Remarks}\label{sec:rem}

The introduction mentioned that algebraic conditions for deciding realizability are known. 
Specifically, given a poset as subsets of a ground set of vertices ordered by containment, the realizability of this poset as a face lattice can be decided using the definition of a face as follows.  Decide if there is a vector for each vertex such that for every subset of vertices, those vectors satisfy the conditions of being a face if and only if they are the vertices of some element of the poset.  This gives the following realizability conditions, which are proved in \cite{grun}. 

\begin{tem}\label{tem:pwrset}
A collection $\mathcal{P}$ of subsets of $J:=\{1,\cdots,n\}$ that includes $\emptyset,\{j\},J\in \mc{P}$ for all $j\in J$ ordered by containment, 
is realizable if and only if there are vectors $w_j\in\mathbb{R}^d$ such that the following holds.  For any proper subset $F\subsetneq J$, there is a vector $h\in\mathbb{R}^d$ with $\left<h,w_j\right> = 1$ for all $j\in F$ and $\left<h,w_j\right> < 1$ for all $j \nin F$ if and only if $F\in \mc{P}$.
\end{tem}

In contrast to Theorem \ref{tem:rmcp}, Theorem \ref{tem:pwrset} imposes essentially no combinatorial restrictions on $\mc{P}$, 
but this is compensated for by more stringent algebraic conditions.  To see how these differ, 
suppose we have found a set of vectors $w_j$ required for Theorem \ref{tem:pwrset}.  Let $W$ be the matrix with columns $w_j$, and let $H$ be the matrix with a column for each maximal proper subset $F_i \in \mc{P}$ given by the corresponding vector $h_i$.  Note that maximal proper subsets of $\mc{P}$ are the subsets of vertices contained in each facet of the resulting polytope.  With this, $H^*W$ is the facet-vertex matrix of the polytope, and the conditions of Theorem \ref{tem:pwrset} require that $H^*W$ be a filled 1-incidence matrix, like in Theorem \ref{tem:rmcp}, and since both $H$ and $W$ have rank at most $d$, $H^*W$ has rank at most $d$. 
Theorem \ref{tem:pwrset}, however, additionally requires us to 
find such vectors $h$ for all faces, and show that no vector $h$ exist for all other subsets of vertices. In prenex normal form, this later condition gives universal quantifiers. 
The algebraic part of Theorem \ref{tem:rmcp}, however, has only existential quantifiers in prenex normal form, which makes this a decision problem in the existential theory of the reals ETR in the number of facets and vertices $n+m$.

Often when we are able to reduce a problem to linear algebra it becomes computationally tractable, but the results shown here do not provide a polynomial time algorithm to determine whether a poset can be realized by a polytope or a cone.  
It is known that finding realizations of polytopes is ETR complete in $n+m$. 
Richter-Gebert showed ETR in $n+m$ is a lower bound for deciding realizability, even for just 4-polytopes, 
by encoding polynomial formulas into a polytope in a combinatorial way that determines its realizations \cite{RG}.  This construction also showed ETR in $|\mc{P}|$, the number of faces, is a lower bound. 
As such, the existence of a polynomial time algorithm for deciding realizability would imply that P=NP=ETR. 
For the upper bound, if we know that a given relation comes from the facets and vertices of a strongly regular spherical complex, then finding a rank $d$ filled 1-incidence matrix is sufficient for realizability, so ETR in $n+m$ is also an upper bound in this case.  Note that a strongly regular spherical complex is a lattice and diamond and flag connected. 
If we are given a general relation, however, the existence of such a matrix is not enough.  The rank of a poset and the diamond condition are easily decided in polynomial time in $|\mc{P}|$ and so is deciding for each face if the graph of faces it covers is connected, which is enough by Theorem \ref{tem:connect}, so polynomial in $|\mc{P}|$ $+$ ETR in $n+m$ is an upper bound for deciding the realizability of a general relation. 
Note that a $d$-simplex, for example, has $n+m = 2d+2$ facets and vertices, $|\mc{P}| = 2^{d+1}$ faces, and $(d+1)!$ flags, each of which grows rapidly in the previous quantity, so the distinction between these can be significant, but an $n$-gon, for example, has $n+m = 2n$, $|\mc{P}| = 2n+2$, and $2n$ flags, so these can be close.



The facet-vertex matrix of a polytope and the facet-ray matrix $N$ of a cone $C$ are related to the Gale dual. 
We can think of $N$ as a linear map from vectors $x \in \mb{R}^m$ of weights associated to each of the cone's generators, which we refer to as formal linear combinations, to $\mb{R}^n$.  If $x$ is in the closed positive orthant and $N=H^*W$ is a decomposition into unit facet normals and generators, $N$ first sends $x$ to a vector $Wx$ in the cone and then sends that to a vector $Nx = H^*Wx$ of distances from each facet.  
For the moment, consider a cone as just a set of points. 
Each point comes from an equivalence class of non-negative formal linear combinations, which are determined by the linear relations (dependencies) of the generators, the same relations as those of the columns of $N$.  
That is, the points of the cone come from $\mb{R}_{\geq 0}^m / \sim $ where $x\sim y$ when $x-y \in R = {\rm null}(N)$, meaning $Wx$ and $Wy$ are the same distance from each facet.  
Allowing negative weights, we can express each vector $v \in \mb{R}^d \supset C$ canonically as the formal linear combination $N^+ v \in \mb{R}^m$ with smallest norm, 
\[ N^+ v = \argmin_{\{x : N x = v \}} \|x\|, \] 
where $N^+$ is the pseudoinverse of $N$. 
This gives us a linear copy of the cone in $\mb{R}^m$ by projecting the positive orthant into the orthogonal compliment of the relations, $C \simeq C_0 = {\rm proj}_{R^\bot} \mb{R}_{\geq 0}^m = N^+N \mb{R}_{\geq 0}^m$, and gives a set of generators of $C_0$ by projecting the standard basis vectors into $R^\bot$.  The generators $x_j = N^+Ne_j$ of $C_0$ are the columns of $N^+N$.  
This also express each generator $w_j = We_j = Wx_j$ of $C$ as a linear combination of all generators of $C$, and as such associates a linear relation to each generator $r_j = e_j-x_j$. These relations $r_j$ are the columns of $I-N^+N$. 
The Gale dual of the cone is the vector arrangement $\{ r_j\}_{j\in[m]} \subset R$ of the relation associated to each generator, which is the orthogonal projection of the standard basis vectors into $R = {\rm null}(N)$.

Analogously in the case of polytopes, the facet-vertex matrix $M$ sends formal weighted averages of the vertices to signed lengths along the outward normals of each facet, 
and we can get the Gale dual of the polytope by either adding one more linear relation $\mb{1}$ to the null space of $M$, which requires the sum of the weights to be 1 for each formal weighted average of vertices, or equivalently take the Gale dual of the cone over the polytope in homogeneous coordinates $C = \mb{R}_{\geq 0} { P \brack 1}$.  
The Gale dual and some of its applications to polytopes are presented in \cite{zieg}.

We presented the Gale dual of a cone by first expressing the generators as a projection of the standard basis vectors of $\mb{R}^{m}$ into a linear subspace $R^\bot$, and then getting the Gale dual as the projection of the standard basis into $R$.  The Gale dual is more commonly presented, however, by extending 
$\mb{R}^{d+1} \supset C$ to $\mb{R}^m$.  In this case, 
instead of projecting the standard basis into $R^\bot$ and $R$ to get the generators and Gale dual, we find some orthonormal set of vectors $q_j=Qe_j$, such that 
projecting $q_j$ to the first $d+1$ coordinates gives the $j^{\rm th}$ generator of the cone, and projecting to the last $m-d-1$ coordinates gives the $j^{\rm th}$ vector of the Gale dual.  This difference amounts to an orthogonal change of coordinates by $Q$. 

If we choose generators and cogenerators of a cone that are unit length and find the full singular value decomposition of the facet-ray matrix $N =U\Sigma V^*$, we get the Gale dual of the cone along with the Gale dual of its polar. 
That is, from the compact singular value decomposition 
extend $U$ and $V$ to full orthogonal matrices and augment $\Sigma$ with $0$. 
The additional columns of $V$ give a $(m-d-1)$-vector for each generator, which is the Gale dual of $C$.  Also, the additional columns of $U$ give a $(n-d-1)$-vector for each facet, which is the Gale dual of the polar $C^*$. 
We also get the Gramian of a cone's generators and cogenerators together from the facet-ray matrix $N$ as follows. 
\[ \tilde{G}:= \left[ \begin{array}{cc}
\sqrt{NN^*} & N \\
N^* & \sqrt{N^*N} \\
\end{array} \right] =\left[ \begin{array}{cc}
U\Sigma U^* & U\Sigma V^* \\
V\Sigma U^* & V\Sigma V^* \\
\end{array} \right] =\left[ \begin{array}{c}
U \\
V \\
\end{array} \right] \Sigma \left[ \begin{array}{cc}
U^* & V^* \\
\end{array} \right]  \]
This perhaps innocuous looking formula relates several objects, a polytopal cone, its Gale dual, its dihedral angles, and its polar cone along with the Gale dual and dihedral angles of the polar. 
the Gramian of the cone is $\sqrt{NN^*}$, and the Gramian of its polar is $\sqrt{N^*N}$. 
That is, the entries of the first diagonal block are the inner products of the outward normals of the cone's facets, which are cosine of the dihedral angles.  Similarly the second diagonal block is that for the polar cone.  
It would be nice if we could add to this a geometric interpretation of the singular values $\Sigma$ of either $N$ or $M$.

\section*{Acknowledgments}

The author would like to thank Andreas Holmsen, Igor Rivin, G\"unter Ziegler, Sinai Robins, and anonymous reviewers for their comments.

This research was supported by NRF grant 2011-0030044 (SRC-GAIA) funded by the government of Korea.

\bibliographystyle{plain}
\bibliography{rmcartbib}

\begin{thebibliography}{10}

\bibitem{introlat}
B.~A. Davey and H.~A. Priestley.
\newblock {\em Introduction to Lattices and Order}.
\newblock Cambridge Math. Textbooks. Cambridge University Press, 1990.

\bibitem{diaz}
Raquel D\'{i}az.
\newblock A characterization of gram matrices of polytopes.
\newblock {\em Discrete Comput. Geom.}, 21:581--601, 1999.

\bibitem{grun}
Branko Gr{\"u}nbaum.
\newblock {\em Convex Polytopes}, chapter 5.5.
\newblock Springer-Verlag, second edition, 2003.

\bibitem{arp}
Peter McMullen and Egon Schulte.
\newblock {\em Abstract Regular Polytopes}, volume~92 of {\em Encycl. of Math.
  and its Applications}.
\newblock Cambridge University Pess, 2002.

\bibitem{paffen}
Andreas Paffenholz.
\newblock {\em Construction for Posets, Lattices, and Polytopes}.
\newblock PhD thesis, Technischen Universit{\"a}t Berlin, 2005.

\bibitem{peterin}
Iztok Peterin.
\newblock Characterizing flag graphs and induced subgraphs of cartesian product
  graphs.
\newblock {\em Order}, 21:283--292, 2004.

\bibitem{RG}
J{\"u}rgen Richter-Gebert.
\newblock {\em Realization Spaces of Polytopes}, volume 1643 of {\em Lecture
  Notes in Math.}
\newblock Springer-Verlag, 1996.

\bibitem{robertson}
Stewart~A. Robertson.
\newblock {\em Polytopes and Symmetry}, volume~90 of {\em London Math. Society
  Lecture Note Series}.
\newblock Cambridge University Pess, 1984.

\bibitem{enumcomb}
Richard~P. Stanley.
\newblock {\em Enumerative Combinatorics Volume I}.
\newblock Wadsworth \& Brooks/Cole Advanced Books \& Software, 1986.

\bibitem{steinitz}
Ernst Steinitz.
\newblock Polyeder und raumeinteilungen.
\newblock {\em Encyclop{\"a}die der math. Wissenschaften}, 3:1--139, 1922.

\bibitem{zieg}
G{\"u}nter~M. Ziegler.
\newblock {\em Lectures on Polytopes}.
\newblock Graduate Texts in Math. Springer-Verlag, 2007.

\end{thebibliography}

\end{document}